\documentclass[graybox,10pt]{svmult}
\pdfoutput=1

\usepackage{makeidx}         
\usepackage{graphicx}        
\usepackage{multicol}        
\usepackage[bottom]{footmisc}

\makeindex             

\usepackage{fancyhdr}
\fancypagestyle{specialfooter}{%
	\fancyhf{}
	
	\fancyfoot[L]{\small Accepted to appear in \emph{Model Reduction of Parametrized Systems III}, Springer.}
}

\usepackage{graphicx}      
\usepackage{caption, subcaption}
\usepackage{amsmath, amssymb, amsfonts, mathtools} 
\usepackage{csquotes}
\usepackage[normalem]{ulem} 
\usepackage[ngerman,english]{babel}
\usepackage{enumerate}
\usepackage{xparse} 
\usepackage{soul}

\usepackage{color}
\definecolor{TUMblau}		{RGB}{0, 101, 189}
\definecolor{TUMgruen}		{RGB}{162, 173, 0}
\definecolor{TUMorange}		{RGB}{227, 114, 34}
\definecolor{TUMdunkelrot}	{RGB}{156,013,022}
\definecolor{gray}			{RGB}{140,140,140}

\usepackage{mcode}

\usepackage[
]{hyperref}

\usepackage{booktabs,units} 

\usepackage{algorithm, algpseudocode} 

\usepackage{pgfplots}
	\pgfplotsset{compat=newest} 
	\pgfplotsset{plot coordinates/math parser=false}
	\pgfplotsset{yticklabel style={text width=2.5em,align=right}}
	\pgfplotsset{/pgf/number format/.cd, fixed,precision=3}
	
	\newlength\myheight
	\newlength\mywidth

\newcommand{\Reals}{\mathbb{R}}
\newcommand{\Complex}{\mathbb{C}}
\newcommand{\Htwo}{\mathcal{H}_2}
\newcommand{\Hinf}{\mathcal{H}_\infty}

\newcommand{\defeq}{\vcentcolon=}

\newcommand{\trans}[1]{#1^\top}

\newcommand{\inv}[1]{#1^{-1}}
\newcommand{\invt}[1]{#1^{-\top}}

\DeclareMathOperator{\diagmat}{diag}
	\newcommand\diag[1]{\diagmat\left(#1\right)}

\newcommand{\norm}[1]{\left\lVert#1\right\rVert}

\newcommand{\corr}[1]{#1} 

\newcommand{\Ltwo}{\mathcal{L}_2}
\newcommand{\Linf}{\mathcal{L}_\infty}
\DeclareMathOperator*{\argmin}{arg\,min}

\newcommand{\fo}{N} 
\newcommand{\ro}{n} 
\newcommand{\nm}{n_m} 

\newcommand{\Dr}{D_r}
\newcommand{\DrOpt}{\Dr^\ast}

\newcommand{\sI}{\sigma} 
\newcommand{\sO}{\mu} 
\newcommand{\rt}{r} 
\newcommand{\lt}{l} 

\newcommand{\Vprim}{\widetilde{V}}
\newcommand{\Wprim}{\widetilde{W}}
\newcommand{\Rprim}{\widetilde{R}}
\newcommand{\Lprim}{\widetilde{L}}

\newcommand{\AllShifts}{\left\{\sI_i \right\}_{i=1}^\ro}
\newcommand{\AllShiftsO}{\left\{\sO_i \right\}_{i=1}^\ro}
\newcommand{\AllRt}{\left\{ \rt_i\right\}_{i=1}^\ro}
\newcommand{\AllLt}{\left\{ \lt_i\right\}_{i=1}^\ro}

\newcommand{\kIrka}{k}
\newcommand{\IRes}{\hat{b}}
\newcommand{\ORes}{\hat{c}}

\newcommand{\GrD}{G_r^D}
\newcommand{\GrDOpt}{G_r^{\ast}}
\newcommand{\GrpD}{\widetilde{G}_{r}^D}   
\newcommand{\ArD}{A_r^D}

\newcommand{\BrD}{B_r^D}
\newcommand{\CrD}{C_r^D}

\newcommand{\Tv}{T_v}
\newcommand{\Tw}{T_w}

\newcommand{\DrZ}{\Dr^0}
\newcommand{\GrZ}{G_r^0}
\newcommand{\GeZ}{G_e^0}

\newcommand{\Arsig}{\mathcal{K}_r}

\newcommand{\Loew}{\mathbb{L}}
\newcommand{\sLoew}{\sigma\mathbb{L}}
\hypersetup{ 
		colorlinks=true,						
		linkcolor=TUMblau,						
		citecolor=TUMdunkelrot,					
		urlcolor=TUMorange,						
		draft = false
	}

\begin{document}

\title*{Interpolatory methods for $\Hinf$ model reduction of multi-input/multi-output systems\thanks{The work of the first author was supported by the German Research Foundation (DFG), Grant LO408/19-1, as well as the TUM Fakult\"{a}ts-Graduiertenzentrum Maschinenwesen; the work of the second author was supported by the Einstein Foundation - Berlin;  the work of the third author was supported by the Alexander von Humboldt Foundation.}
	}
\titlerunning{Interpolatory $\Hinf$ model reduction of MIMO systems}
\author{Alessandro Castagnotto, Christopher Beattie, and Serkan Gugercin}
\institute{Alessandro Castagnotto \at Technical University of Munich, Garching bei M\"{u}nchen \email{a.castagnotto@tum.de} 
\and Christopher Beattie \at Virginia Tech, Blacksburg, VA  24061 USA  \email{beattie@vt.edu}
\and Serkan Gugercin \at Virginia Tech, Blacksburg, VA  24061 USA  \email{gugercin@vt.edu}}
\maketitle

\thispagestyle{specialfooter}

\abstract{We develop here a computationally effective approach for producing high-quality $\Hinf$-approximations to large scale linear dynamical systems having multiple inputs and multiple outputs (MIMO). 
We extend an approach for $\Hinf$ model reduction introduced by Flagg, Beattie, and Gugercin \cite{Flagg_2013_Hinf} for the single-input/single-output (SISO) setting, which combined ideas originating in interpolatory $\Htwo$-optimal  model reduction with complex Chebyshev approximation.  
Retaining this framework, our approach to the MIMO problem has its principal computational cost dominated by (sparse) linear solves, and so it can remain an effective strategy in many large-scale settings.
\textcolor{black}{We are able to avoid computationally demanding $\Hinf$ norm calculations that are normally required to monitor progress 
 within each optimization cycle through the use of  ``data-driven'' rational approximations that are built upon previously computed function samples}. 
Numerical examples are included that illustrate our approach.  We produce high fidelity reduced models having consistently better $\Hinf$ performance than models produced via balanced truncation; 
these models often are as good as (and occasionally better than) models produced using optimal Hankel norm approximation as well. 
In all cases considered, the method described here produces reduced models at far lower cost than is possible with either balanced truncation or optimal Hankel norm approximation.}


\section{Introduction}
\label{sec:Intro}

The accurate modeling of dynamical systems often requires that a large number of differential equations describing the 
evolution of a large number of state variables be integrated over time to predict system behavior.  
The number of state variables and differential equations involved can be especially large and forbidding when these models arise, say, from a modified nodal analysis of integrated electronic circuits, or more broadly, from a spatial discretization of partial differential equations over a fine grid.  Most dynamical systems arising in practice can be represented at least locally around an operating point, with a state-space representation having the form
\begin{equation}
\begin{aligned}
E \,\dot{x} &= A\, x + B\, u, \\
y &= C\, x + D\, u,
\end{aligned}
\label{eq:FOM}
\end{equation}	
where $E\ts\in\ts\Reals^{\fo \ts\times\ts \fo}$ is the \emph{descriptor matrix}, $A\ts\in\ts\Reals^{\fo\ts\times\ts\fo}$ is the system matrix and $x\ts\in\ts\Reals^\fo$, $u\ts\in\ts\Reals^m$, and $y\ts\in\ts\Reals^p$ ($p,m\ts\ll\ts \fo$) represent the state, input, and output of the system, respectively.  A static feed-through relation from the control input $u$ to the control output $y$ is modeled through the matrix $D\in\Reals^{p\times m}$.   Most practical systems involve several actuators (input variables) and several quantities of interest (output variables), motivating our focus here on systems having multiple inputs and multiple outputs (MIMO).

In many application settings, the state dimension $\fo$ (which typically matches the order of the model) can grow quite large as greater model fidelity is pursued, and in some cases it can reach magnitudes of $10^6$ and more.  Simulation,  optimization, and control design based on such large-scale models becomes computationally very expensive, at times even intractable. This motivates consideration of \emph{reduced order models} (ROMs), which are comparatively low-order models that in spite of having significantly smaller order, $\ro \ll \fo$, are designed so as to reproduce the input-output response of the full-order model (FOM) accurately while preserving certain fundamental structural properties, that may include stability and passivity.
For state space models such as \eqref{eq:FOM}, reduced models are obtained generally through Petrov-Galerkin projections having the form:
\begin{equation}\label{eq:ROM}	
\begin{aligned}
\overbrace{\trans{W}E\,V}^{E_r} \, \dot{x}_r\, &= \, \overbrace{\trans{W}A\,V}^{A_r}\, x_r \, +  \, \overbrace{\trans{W} B}^{B_r} \, u,\\ 
y_r \, &= \; \underbrace{C\,V}_{C_r}\, x_r \,+ \; D_r \, u.
\end{aligned}
\end{equation}
The projection matrices $V,W \in \Reals^{\fo\ts\times\ro}$ become the primary objects of scrutiny in the model reduction enterprise, since how they are chosen has a great impact on the quality of the ROM.
For truly large-scale systems, \emph{interpolatory model reduction}, which includes approaches known variously as \emph{moment matching} methods and \emph{Krylov subspace} methods, has drawn significant interest due to its flexibility and comparatively low computational cost \cite{Antoulas_Book,Gallivan_2002,Beattie_2014_Survey}.   Indeed, these methods typically require only the solution of large (generally sparse) linear systems of equations, for which several optimized methods are available. Through the appropriate selection of $V$ and $W$, it is possible to match the action of the transfer function
\begin{equation}  \label{eq:FOMTransFnc}
G(s) = C\inv{\left(sE-A\right)}B + D
\end{equation}
along arbitrarily selected \emph{input} and \emph{output} tangent directions at arbitrarily selected (driving) frequencies. 
The capacity to do this is central to our approach and is stated briefly here as: 
\begin{theorem}[\cite{Grimme_PhD, Gallivan_2004_MIMO}]
	Let $G(s)$ be the transfer function matrix \eqref{eq:FOMTransFnc} of the FOM \eqref{eq:FOM} and let $G_r(s)$ be the transfer function matrix of an associated ROM obtained through Petrov-Galerkin projection as in \eqref{eq:ROM}. 
	Suppose $\sI,\sO\in\Complex$ are complex scalars (``shifts") that do not coincide with any eigenvalues of the matrix pencil $(E,A)$ but otherwise are arbitrary. Let also $\rt\in\Complex^m$ and $\lt\in\Complex^p$ be arbitrary nontrivial tangent directions. 
	Then
	\begin{subequations}\label{eq:interpolation}
\begin{align}
		G(\sI)\cdot \rt = G_r(\sI)\cdot \rt  \quad& \text{ if } \inv{\left(A - \sI E\right)}B\rt \in \mathsf{Ran}(V),\label{eq:HI1}\\
		\trans{\lt}\cdot G(\sO) = \trans{\lt}\cdot G_r(\sO) \quad&\text{ if } {\invt{\left(A - \sO E\right)}\trans{C}\lt \in \mathsf{Ran}(W)}, \\
 \trans{\lt}\cdot G'(\sI)\cdot \rt = \trans{\lt} \cdot G'_r(\sI) \cdot \rt \quad &\text{ if, additionally, }  \sI=\sO.
		\end{align}
	\end{subequations}
\end{theorem}
A set of complex shifts, $\AllShifts$, $\AllShiftsO$, with corresponding tangent directions, 
$\AllRt$, $\AllLt$, will be collectively referred to as 
\emph{interpolation data} in our present context.  We define \emph{primitive projection matrices} as
\begin{subequations}%
	\begin{align} %
	\Vprim &\defeq \left[\inv{(A-\sI_1 E)}B\rt_1, \dots, \inv{(A-\sI_\ro E)}B\rt_\ro\right] \label{eq:Vprim}\\
	\Wprim &\defeq \left[\invt{(A-\sO_1 E)}\trans{C}\lt_1, \dots, \invt{(A-\sO_\ro E)}\trans{C}\lt_\ro\right] \label{eq:Wprim}
	\end{align}\label{eq:VWprim}%
\end{subequations}%
Note that $\Vprim$ and $\Wprim$ satisfy Sylvester equations having the form:	
\begin{equation}%
	A\,\Vprim - E\,\Vprim S_{\sI} = B \Rprim \quad \mbox{and}\quad
	\trans{A}\Wprim - \trans{E}\Wprim\,\trans{S_{\sO}} = \trans{C} \Lprim, \label{eq:Sylvester}%
\end{equation}%
where $S_{\sI} =  \diag{\sI_1,..,\sI_\ro}\in\Complex^{\ro\times\ro}$, $S_{\sO} = \diag{\sO_1,..,\sO_\ro}\in\Complex^{\ro\times\ro}$, $\Rprim = \left[\rt_1,..,\rt_\ro\right]\in\Complex^{m\times\ro}$ and $\Lprim= \left[\lt_1,\dots,\lt_\ro\right]\in\Complex^{p\times\ro}$ \cite{Gallivan_2004_Sylvester}.  In this way, the Petrov-Galerkin projection of \eqref{eq:ROM} is parameterized by interpolation data and the principal task in defining interpolatory models then becomes the judicious choice of shifts and tangent directions.

 Procedures have been developed over the past decade for choosing interpolation data that yield reduced models, $G_r(s)$, that minimize, at least locally the approximation error, $G(s)-G_r(s)$, as measured with respect to the $\mathcal{H}_2$-norm: 
\begin{equation}
\norm{\corr{G-G_r}}_{\Htwo} \defeq \sqrt{\frac{1}{2\pi}\int_{-\infty}^{\infty}\norm{G(j\omega)-G_r(j\omega)}^2_Fd\omega}
\end{equation}
(see \cite{Antoulas_Book}).
Minimizing the $\Htwo$-error, $\norm{G-G_r}_{\Htwo}$, is of interest through the immediate relationship this quantity bears with the induced system response error:
\begin{equation}
\norm{y-y_r}_{\Linf} \leq  \norm{G-G_r}_{\Htwo}\norm{u(t)}_{\Ltwo},
\end{equation}
%
A well-known approach to accomplish this that has become popular at least in part due to its simplicity and effectiveness is the \emph{Iterative Rational Krylov Algorithm} (IRKA) \cite{Gugercin_2008_IRKA}, which\corr{, in effect}, runs  a simple fixed point iteration aimed at producing interpolation data that satisfy 
\textcolor{black}{first-order $\Htwo$-optimality conditions,}  
 i.e.,
\begin{subequations}\label{eq:Meier-Luenberger}
	\begin{align}
	G(-\lambda_i)\cdot\IRes_i = G_r(-\lambda_i)\cdot \IRes_i, \quad & 
	\trans{\ORes_i}\cdot G(-\lambda_i) = \trans{\ORes_i}\cdot G_r(-\lambda_i), \\
\mbox{and}\quad	\trans{\ORes_i}\cdot G'(-\lambda_i)\cdot\IRes_i &= \trans{\ORes_i}\cdot G_r'(-\lambda_i)\cdot\IRes_i.
	\end{align}
\end{subequations} 
for $i=1,\dots,\ro$.
The data $\lambda_i$, $\IRes_i$ and $\ORes_i$ are reduced poles and right/left vector residues corresponding to the pole-residue expansion of the ROM:
\begin{equation}\label{eq:PoleResidue}
G_r(s) = \sum_{i=1}^{\ro} \frac{\ORes_i\,\trans{\IRes_i}}{s-\lambda_i}.
\end{equation}

Despite the relative ease with which $\mathcal{H}_2$-optimal reduced models can be obtained, there are several circumstances in which it might be preferable to obtain a ROM which produces a small error as measured in the $\Hinf$-norm: 
\begin{equation}
\norm{G-G_r}_{\Hinf} \defeq \max_\omega\,\varsigma_{max}(G(j\omega)-G_r(j\omega)),
\end{equation}
where $\varsigma_{max}(M)$ denotes the largest singular value of a matrix $M$ (see \cite{Antoulas_Book}).
ROMs having small $\Hinf$-error produce an output response with a uniformly bounded ``energy" error:
 \begin{equation}
\norm{y-y_r}_{\Ltwo} \leq \norm{\corr{G-G_r}}_{\Hinf} \norm{u}_{\Ltwo}.
\end{equation}
The $\Hinf$-error is also used as a robustness measure for closed-loop control systems and is therefore of central importance in robust control.  It finds frequent use in aerospace applications, among others, \corr{where the $\Ltwo$ energy of the system response is of critical interest in design and optimization.}

Strategies for producing reduced models that give good $\Hinf$ performance has long been an active area of research \cite{antoulas2002h}.  Analogous to the $\Hinf$-control design problem, the optimal $\Hinf$ reduction problem can be formulated in terms of \emph{linear matrix inequalities}, although advantageous features such as linearity and convexity are lost in this case \cite{helmersson1994model,varga2001fast}. Due to the high cost related to solving these matrix inequalities, this approach is generally not feasible in large-scale settings. 

Another family of methods for the $\Hinf$ reduction problem relates it to the problem of finding an \emph{optimal Hankel norm approximation} (OHN\corr{A}) \textcolor{black}{\cite{Kavranoglu_1993,Glover_1984,trefethen1981rational}}.  Along these lines the \emph{balanced truncation} (BT) algorithm yields rigorous upper bounds on the $\Hinf$ error and often produces small approximation error, especially for higher reduced order approximants \cite{Antoulas_Book, gugercin2004survey}. Each of these procedures is generally feasible only for mid-size problems since either an all-pass dilation requiring large-scale eigenvalue decomposition 
(for OHNA) or the solution of generalized Lyapunov equations (for BT) is required. Extensions to large-scale models are available, however -- e.g., in 
\textcolor{black}{\cite{benner2004computing,Li_2000,benner2014self,kurschner2016efficient,sabino2006solution,gugercin2003modified}}.

A wholly different approach to the $\Hinf$ model reduction problem for SISO models was proposed by Flagg, Beattie, and Gugercin in \cite{Flagg_2013_Hinf}.  A locally $\mathcal{H}_2$-optimal reduced model is taken as a starting point and adjusted through the variation of rank-one modifications parameterized by the scalar feed-through term, $D$.  Minimization of the $\Hinf$-error with respect to this parameterization available through $D$ produces ROMs that are observed to have generally very good $\Hinf$-performance, often exceeding what could be attained with \corr{OHNA}. 

In this work, we extend these earlier interpolatory methods to MIMO systems. We introduce a strategy that reduces the computational expense of the intermediate optimization steps by means of data-driven MOR methods (we use \emph{vector fitting} \textcolor{black}{\cite{gustavsen1999rational,drmac2015vectorrohtua} }\corr{)}. 
Stability of the reduced model is guaranteed through appropriate constraints in the resulting multivariate optimization problem. Numerical examples show effective reduction of approximation error, often outperforming both \corr{OHNA} and BT.

\section{MIMO Interpolatory $\Hinf$-approximation (MIHA)}

\corr{In this section we first characterize the $\Hinf$-optimal reduced order models from the perspective of rational interpolation. This motivates the usage of $\Htwo$-optimal reduction as a starting point for the model reduction algorithm we propose for  the $\Hinf$ approximation problem.}

\subsection{Characterization of $\Hinf$-approximants via rational interpolation}

In the SISO case, Trefethen \textcolor{black}{\cite{trefethen1981rational}} has characterized best $\Hinf$ approximations within a broader context of rational interpolation:  
\begin{theorem}[Trefethen \cite{trefethen1981rational}] \label{thm:trefsufficient}
Suppose $G(s)$ is a (scalar-valued) transfer function associated with a SISO dynamical system as in (\ref{eq:FOMTransFnc}).  
Let $\widehat{G}_r(s)$ be an optimal $\Hinf$ approximation to $G(s)$ and let $G_r$  be any $n^{\rm th}$ order stable approximation to $G(s)$ that interpolates $G(s)$ at $2n+1$ points in the open right  half-plane. 
Then
$$
\min_{\omega \in \Reals} | G(j \omega) - G_r(j \omega)| \leq  \| G - \widehat{G}_r \|_{\Hinf} \leq \| G - G_r \|_{\Hinf} 
$$
In particular,  if
$|G(j \omega)-G_r(j \omega)|=\mathsf{const}$ for all $\omega\in \Reals $ 
then $G_r$ is itself an optimal $\Hinf$-approximation to $G(s)$.	
\end{theorem}

For the SISO case, a good $\Hinf$ approximation will be obtained when the modulus of the error, 
$|G(s)-G_r(s)|$, is nearly constant as $s=j \omega$ runs along the imaginary axis. 
In the MIMO case, the analogous argument becomes more technically involved as the maximum singular value of matrix-valued function $G(s)-G_r(s)$ will not generally be analytic in the neighborhood of the imaginary axis (e.g., where multiple singular values occur).  Nonetheless, the intuition of the SISO case carries over to the MIMO case, as the following \emph{Gedankenexperiment} might suggest: Suppose 
that $\widehat{G}_r$ is an $\Hinf$-optimal interpolatory approximation to $G$ but $\varsigma_{max}(G(j \omega)-G_r(j \omega))$ is not constant with respect to $\omega\in\Reals$. Then there exist frequencies $\hat{\omega}$ and $\tilde{\omega}\in\Reals$ and $\epsilon>0$ such that
\begin{align*}
\|G-\widehat{G}_r\|_{\Hinf}=\varsigma_{max}(G(j \hat{\omega})-\widehat{G}_r(j \hat{\omega}))  \geq &
\, \epsilon+ \min_{\omega}\varsigma_{max}(G(j\omega)-\widehat{G}_r(j\omega))  \\
& =\epsilon+ \varsigma_{max}(G(j\tilde{\omega})-\widehat{G}_r(j\tilde{\omega})).
\end{align*}
By nudging interpolation data away from the vicinity of $\tilde{\omega}$ and toward $\hat{\omega}$ 
while simultaneously nudging the poles of $\widehat{G}_r$ away from the vicinity of $\hat{\omega}$ and toward
$\tilde{\omega}$, one may decrease the value of $\varsigma_{max}(G(j \hat{\omega})-\widehat{G}_r(j \hat{\omega}))$ while increasing the value of $\varsigma_{max}(G(j\tilde{\omega})-\widehat{G}_r(j\tilde{\omega}))$. This will (incrementally) decrease the $\Hinf$ norm and bring the values of $\varsigma_{max}(G(j \hat{\omega})-\widehat{G}_r(j \hat{\omega}))$ and $\varsigma_{max}(G(j\tilde{\omega})-\widehat{G}_r(j\tilde{\omega}))$ closer together toward a common value. 

Of course, the nudging process described above contains insufficient detail to suggest an algorithm, and indeed, our approach to this problem follows a somewhat different path, a path that nonetheless uses the guiding heuristic for (near) $\Hinf$-optimality:
\begin{equation}\label{circErrorCond}
\varsigma_{max}(G(j \omega)-\widetilde{G}_r(j \omega))\approx\mathsf{const}\quad\mbox{ for all }\omega\in\Reals.
\end{equation}
Approximations with good $\Hinf$ performance should 
have an advantageous configuration of poles and interpolation data 
that locates them symmetrically about the imaginary axis, thus balancing regions where $\varsigma_{max}(G(s)-\widetilde{G}_r(s))$ is big (e.g., pole locations) symmetrically against regions reflected across the imaginary axis where $\varsigma_{max}(G(s)-\widetilde{G}_r(s))$ is small (e.g., interpolation locations). 
This configuration of poles and interpolation data, we note, is precisely the outcome of optimal $\Htwo$ approximation as well, and this will provide us with an easily computable approximation that is likely to have good $\Hinf$ performance.  

	\subsection{\corr{$\Hinf$ approximation with interpolatory $\Htwo$-optimal initialization}}

	\corr{Local $\Htwo$-optimal ROMs are often observed to give good $\Hinf$ performance -- this is in addition to the expected good $\Htwo$ performance.}
	This $\Hinf$ behaviour is illustrated in Figure \ref{fig:BTvsIRKA}, where the $\Hinf$ approximation 
	errors of  local $\Htwo$-optimal ROMs produced by IRKA are compared to ROMs of the same order obtained through BT for the CD player MIMO benchmark model \cite{Chahlaoui_Benchmark}.		
		\setlength\mywidth{0.8\textwidth}
		\setlength\myheight{3 cm} 
		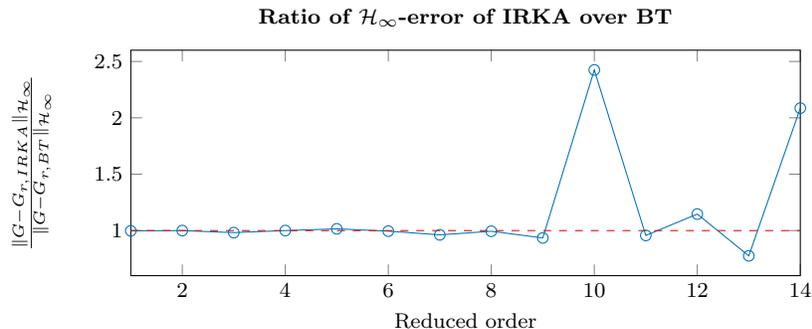
\begin{figure}[h!]
			\centering
%
\definecolor{mycolor1}{rgb}{0.00000,0.44700,0.74100}%
\begin{tikzpicture}

\begin{axis}[%
width=0.95092\mywidth,
height=\myheight,
at={(0\mywidth,0\myheight)},
scale only axis,
xmin=1,
xmax=14,
xlabel={Reduced order},
ymin=0.6,
ymax=2.6,
ylabel={$\frac{\|G-G_{r,IRKA}\|_{\mathcal{H}_\infty}}{\|G-G_{r,BT}\|_{\mathcal{H}_\infty}}$},
title style={font=\bfseries},
title={Ratio of $\mathcal{H}_\infty$-error of IRKA over BT},
legend style={font=\footnotesize}
]
\addplot [color=mycolor1,solid,mark=o,mark options={solid},forget plot]
  table[row sep=crcr]{%
1	0.997851581762705\\
2	0.999999999989451\\
3	0.981759734117045\\
4	1.00000155158615\\
5	1.01521204556222\\
6	0.995085630287912\\
7	0.962076874857423\\
8	0.993911839425368\\
9	0.934882720516036\\
10	2.42489870260363\\
11	0.956640909316975\\
12	1.14637931188759\\
13	0.775836313550081\\
14	2.08591691163017\\
};
\addplot [color=red,dashed,forget plot]
  table[row sep=crcr]{%
1	1\\
14	1\\
};
\end{axis}
\end{tikzpicture}%
			\caption{Numerical investigations indicate that IRKA models are often good also in terms of the $\Hinf$-error.}
			\label{fig:BTvsIRKA}
		\end{figure}
		
The frequently favourable $\Hinf$ behaviour of IRKA models has particular significance in this context, since they are computationally cheap to obtain even in large-scale settings, indeed often they are much cheaper than comparable BT computations.   The resulting locally $\Htwo$-optimal ROMs can be further improved 
(with respect to $\Hinf$ error) by relaxing the (implicit) interpolation constraint at 
$\infty$ while preserving the $\Htwo$-optimal interpolation conditions 
(which is the most important link the $\Htwo$-optimal ROM has with the original model).

Consider the partial fraction expansion 
		\begin{equation}\label{eq:PoleResidueD}
		G_r(s) = \sum_{i=1}^{\fo} \frac{\ORes_i\trans{\IRes_i}}{s-\lambda_{i}} + D_r.
		\end{equation}
For ease of exposition, we assume the poles, $\lambda_{i}$, to be simple, 
although the results we develop here can be extended to the case of higher multiplicity.
The input/output behavior is determined by $\ro$ scalar parameters $\lambda_{i}$, $\ro$ pairs of input/output residuals $\IRes_i,\ORes_i$ and the $p\times m$-dimensional feed-through $D_r$. Considering that a constant scaling factor can be arbitrarily defined in the product of the residuals, this leaves us a total of $\ro\left(p+m\right) + p\cdot m$ parameters, $\ro\left(p+m\right)$ of which can be described in terms of two-sided tangential interpolation conditions \eqref{eq:interpolation}.   This interpolation data is established for the original $\Htwo$-optimal ROM and we wish it to remain invariant over subsequent adjustments, so the only remaining degrees-of-freedom are the $p\cdot m$ entries in the feed-through matrix $D_r$.
		
In the typical context of $\Htwo$-optimal model reduction, $D_r$ is chosen to match the feed-through term $D$ of the original model, thus guaranteeing that the error $G-\widetilde{G}_r$ remains in $\Htwo$.  Note that $D$ remains untouched by the state-space projections in \eqref{eq:ROM},  moreover since typically $p,m \ll \fo$, the feed-through term need not be involved in the reduction process and may be retained from the \corr{FOM}.  Indeed,  retaining the original feed-through term is a necessary condition for $\Htwo$ optimality, forcing interpolation at $s=\infty$ and as a consequence, small error at higher frequencies.
		Contrasting significantly with $\Htwo$-based model reduction, good $\Hinf$ performance does not require $D_r=D$, and in this work we exploit this flexibility in a crucial way.  
		A key observation playing a significant role in what follows was made in \cite{mayo2007framework,beattie2009interpolatory} that the feed-through term $D_r$ induces a parametrization of all reduced order models satisfying the two-sided tangential interpolation conditions. 
	
		This result is summarized by following theorem taken from \cite[Thm. 4.1]{mayo2007framework} and \cite[Thm. 3]{beattie2009interpolatory}		
		\begin{theorem}\label{th:ROM_Dr}
			Let $\Rprim$, $\Lprim$ be defined through the Sylvester equations in \eqref{eq:Sylvester}. Assume, without loss of generality, that the full order model satisfie\corr{s} $D=0$ and let the nominal reduced model $\GrZ(s)=C_r\inv{\left(sE_r - A_r\right)}B_r$ be obtained through Petrov-Galerkin projection using the primitive projection matrices \eqref{eq:VWprim}. Then, for any $D_r \in \Complex^{p\times m}$, the perturbed reduced order model 
			\begin{equation}\label{eq:ROM_Dr_prim}
			\GrpD(s,\Dr) = \left(\widetilde{C}_{r} + D_r \Rprim \right) \inv{\left[s\widetilde{E}_{r} - \left(\widetilde{A}_{r} + \trans{\Lprim}D_r\Rprim\right) \right]}\left(\widetilde{B}_{r} + \trans{\Lprim}D_r \right) + D_r
			\end{equation}
			also satisfies the tangential interpolation conditions \eqref{eq:interpolation}.
		\end{theorem}

		Note that for $D\neq 0$, the results of Theorem \ref{th:ROM_Dr} can be trivially extended by adding $D$ to the right-hand side in \eqref{eq:ROM_Dr_prim}.
		Even though for theoretical consideration the use of primitive Krylov bases $\Vprim,\Wprim$ introduced in \eqref{eq:VWprim} is often convenient, from a numerical standpoint there are several reason why one may choose a different basis for the projection matrices. This next result shows 
		that the interpolation conditions are preserved also for arbitrary bases---in particular also real and orthonormal bases---provided that the shifting matrices $R$ and $L$ are appropriately chosen.
		\begin{corollary}\label{th:Dr-Shift-Extended}
			Let $\Tv, \Tw \in\Complex^{\ro\times\ro}$ be invertible matrices used to transform the primitive bases $\Vprim, \Wprim$ of the Krylov subspace to  new bases $V = \Vprim \Tv$ and $W = \Wprim\Tw$. Let the same transformation be applied to the matrices of tangential directions, resulting in $R = \Rprim\Tv$ and $L = \Lprim\Tw$. Then, for any $D_r$, the ROM $\GrD$ is given by
			\begin{equation}\label{eq:ROM_Dr}
			\GrD(s,\Dr) = \underbrace{\left(C_r + D_r R \right)}_{\CrD} \inv{\left[sE_r - \underbrace{\left(A_r + \trans{L}D_r R\right)}_{\ArD} \right]}\underbrace{\left(B_r + \trans{L}D_r \right)}_{\BrD} + D_r
			\end{equation}	
		\end{corollary}
		\begin{proof}
			The proof amounts to showing that the transfer function matrix $\GrD$ of the ROM is invariant to a change of basis from $\Vprim$ and $\Wprim$ as long as $\Rprim$ and $\Lprim$ are adapted accordingly. 
			{\small
			\begin{equation*}
			\begin{split}
			&\GrD - D_r = \CrD \inv{\left( s E - \ArD \right)} \BrD \\
			&= \left(CV + \Dr R\right)\inv{\left[ s \trans{W} E V - \trans{W} A V - \trans{L} \Dr \trans{R} \right]} \left(\trans{W} B + \trans{L} \Dr\right) \\
			&= \left(C\Vprim + \Dr \Rprim\right)\Tv\inv{\left[ \trans{\Tw} \left(s \trans{\Wprim} E \Vprim - \trans{\Wprim} A \Vprim - \trans{\Lprim} \Dr \trans{\Rprim} \right) \Tv \right]} \trans{\Tw}\left(\trans{\Wprim} B + \trans{\Lprim} \Dr\right)\\
			&= \left(C\Vprim + \Dr \Rprim\right)\inv{\left[ s \trans{\Wprim} E \Vprim - \trans{\Wprim} A \Vprim - \trans{\Lprim} \Dr \trans{\Rprim} \right]}\left(\trans{\Wprim} B + \trans{\Lprim} \Dr\right) \\
			&= \GrpD - D_r\, \corr{.}
			\end{split}
			\end{equation*}
		}
			The results of Theorem \ref{th:ROM_Dr} generalize to the case of arbitrary bases.
			Following the notation from \cite[Definition 2.1]{mayo2007framework}, the state-space models resulting from Petrov-Galerkin projections with $V,W$ and $\Vprim,\Wprim$ respectively are \emph{restricted system equivalent}. As a consequence, they share the same transfer function matrix.
		\end{proof}
		
		Using the Sherman-Morrison-Woodbury formula \cite{golub2012matrix} for the inverse of rank $k$ perturbations of a matrix\corr{,} we are able to decompose the transfer function of the shifted reduced model into the original reduced model and an additional term.
		\begin{corollary}
			Define the auxiliary variable $\corr{\Arsig}\defeq s E_r - A_r$. The transfer function of the shifted reduced model $\GrD$ can be given as 
			\begin{equation}\label{eq:GrDSplit}
			\GrD(s) = \GrZ(s) + \Delta\GrD(s,\Dr),
			\end{equation}
			where $\GrZ$ is the transfer function of the unperturbed model and $\Delta\GrD$ is defined as
			\begin{equation}
			\begin{aligned}
			\Delta\GrD = \Delta_1 + &\Delta_2 + \Delta_3\cdot\inv{\left(\Delta_4\right)}\cdot\Delta_2 + D_r \\
			\text{given}& \\
			\begin{array}{l}
			\Delta_1 \defeq C_r \inv{\Arsig} \trans{L}D_r \\[2mm]
			\Delta_3 \defeq \left(C_r + D_r R\right) \inv{\Arsig}\trans{L}
			\end{array}
			& \quad
		        \begin{array}{l}
                         \Delta_2 \defeq D_r R \inv{\Arsig} \left(B_r + \trans{L}D_r\right) \\[2mm]
			\Delta_4 \defeq I - D_r R \inv{\Arsig} \trans{L}
		         \end{array}
			\end{aligned}
			\end{equation}
		\end{corollary}
		\begin{proof}
			Note that by the Sherman-Morrison-Woodbury formula, following equality holds:
			\begin{equation}
			\inv{\left(\Arsig - \trans{L}D_rR\right)} = \inv{\Arsig} \corr{+} \inv{\Arsig}\trans{L}\inv{\left(I - D_r R \inv{\Arsig}\trans{L}\right)}D_r R \inv{\Arsig}.
			\end{equation}
			Using this relation in the definition of $\GrD$, the proof is completed by straightforward algebraic manipulations.
		\end{proof}

		We proceed by attempting to exploit the additional degrees-of-freedom available in $D_r$ to trade off excessive accuracy at high frequencies for improved approximation in lower frequency ranges, 
		as measured with the $\Hinf$-norm.
We first obtain an $\Htwo$-optimal ROM by means of IRKA and subsequently minimize the $\Hinf$-error norm with respect to the constant feed-through matrix $D_r$ while preserving tangential interpolation and guaranteeing stability. The resulting ROM $\GrDOpt$ will represent a local optimum out of the set of all stable ROMs satisfying the tangential interpolation conditions.
		The outline of our proposed reduction procedure, called \emph{MIMO interpolatory $\Hinf$-approximation} (MIHA), is given in Algorithm \ref{algo:Hinf}.
			\begin{algorithm}
				\caption{\corr{MIMO Interpolatory $\Hinf$-Approximation (MIHA)}} \label{algo:Hinf}
				\begin{algorithmic}[1]
					\Require $G(s)$, $\ro$
					\Ensure Stable, locally optimal reduced oder model $\GrDOpt$, approximation error $e^{\ast}_{\Hinf}$
					\State {$G_r^0 \gets \text{IRKA}(G(s),\ro)$
					}%
					\State {$\DrOpt \gets \argmin_{D_r} \norm{G(s)-\GrD(s,\Dr)}_{\Hinf}$ s.t. $\GrD(s,\DrOpt)$ is stable} 
					\State {$\GrDOpt \gets \GrD(s,\DrOpt)$}
					\State {$e^{\ast}_{\Hinf} \gets  \norm{G(s)-\GrDOpt(s)}_{\Hinf}$}
				\end{algorithmic}
			\end{algorithm}	
			
		Numerical results in Section \ref{sec:Results} will show the effectiveness of this procedure in further reducing the $\Hinf$-error for a given IRKA model. However, at this stage the optimization in Step 2 appears problematic, for it requires both the computation of the $\Hinf$-norm of a large-scale model and a constrained multivariate optimization of a non-convex, non-smooth function.
		It turns out that both of these issues can be resolved effectively, as it will be discussed in the following sections.			
			
	\subsection{Efficient implementation}
	As we have noted, the main computational burden of the algorithm described above 
	resides mainly in Step 2.  We are able to lighten this burden somewhat through judicious use of  \eqref{eq:GrDSplit} and by taking advantage 
	of previously computed transfer function evaluations.
	
\subsubsection{A ``free'' surrogate model for the approximation error $G-\GrZ$} \label{sec:surrogate}
			Step 1 of Algorithm \ref{algo:Hinf} requires performing $\Htwo$-optimal reduction using  IRKA. This is a fixed point iteration involving a number of steps $\kIrka$ before convergence is achieved. At every step $j$, Hermite tangential interpolation about some complex frequencies $\AllShifts$ and tangential directions $\AllRt$, $\AllLt$ is performed. 
			For this purpose, the projection matrices in \eqref{eq:VWprim} are computed, and it is easy to see that for all $i=1,\dots,\ro$ it holds
			\begin{subequations}
				\begin{align}
					C \cdot \Vprim e_i = C \inv{\left(A - \sI_i E\right)} B r_i  &= G(\sI_i) \rt_i\\
					\trans{e_i} \trans{\Wprim} \cdot B = \trans{\lt_i} C \inv{\left(A - \sI_i E\right)} B &= \trans{\lt_i} G(\sI_i) \\
					\trans{e_i} \trans{\Wprim} E \Vprim e_i = \trans{\lt_i}\inv{\left(A - \sI_i E\right)}E\inv{\left(A - \sI_i E\right)} \rt_i &= \trans{\lt_i} G'(\sI_i) \rt_i
				\end{align}\label{eq:VW2Data}
			\end{subequations}
			Observe that, at basically no additional cost, we can gather information about the FOM while performing IRKA.	
			Figure \ref{fig:IRKAshifts} illustrates this point by showing the development of the shifts during the IRKA iterations reducing the CDplayer benchmark model to a reduced order $\ro=10$.  For all complex frequencies indicated by a marker, tangent data for the full order model is collected.
			
			\setlength\mywidth{0.35\textwidth}
			\setlength\myheight{0.25\textwidth} 
			\begin{figure}
				\centering
				\begin{subfigure}[t]{0.48\textwidth}
%
\begin{tikzpicture}

\begin{axis}[%
width=0.95092\mywidth,
height=\myheight,
at={(0\mywidth,0\myheight)},
scale only axis,
xmin=0,
xmax=400,
xlabel={real},
ymin=-400,
ymax=400,
ylabel={imaginary},
title style={font=\bfseries},
title={Shift development over IRKA iterations},
legend style={legend cell align=left,align=left,draw=white!15!black},
legend style={font=\footnotesize}
]
\addplot [color=blue,only marks,mark=o,mark options={solid}]
  table[row sep=crcr]{%
0	-300\\
0	300\\
0	-225\\
0	225\\
0	-150\\
0	150\\
0	-75\\
0	75\\
0	0\\
0	0\\
};
\addlegendentry{start};

\addplot [color=black,only marks,mark=+,mark options={solid}]
  table[row sep=crcr]{%
12.2770146888159	-306.540779142107\\
12.2770146888159	306.540779142107\\
19.6443648407331	-196.469480294748\\
19.6443648407331	196.469480294748\\
71.9710131308956	173.866826250016\\
71.9710131308956	-173.866826250016\\
0.22617805801485	-22.5683832701212\\
0.22617805801485	22.5683832701212\\
8.7319776499028	-76.3189433071418\\
8.7319776499028	76.3189433071418\\
366.450860600126	0\\
12.2700722933748	-306.540021891128\\
12.2700722933748	306.540021891128\\
147.829298727982	0\\
0.225705552665523	-22.5693361276167\\
0.225705552665523	22.5693361276167\\
8.25492268605781	-76.7763607133495\\
8.25492268605781	76.7763607133495\\
19.750262444193	-196.568804839882\\
19.750262444193	196.568804839882\\
12.2721262068179	-306.541571950942\\
12.2721262068179	306.541571950942\\
19.7461626101327	-196.626028311155\\
19.7461626101327	196.626028311155\\
8.46269203873	-76.8257398779357\\
8.46269203873	76.8257398779357\\
0.225705925864469	-22.5693407651544\\
0.225705925864469	22.5693407651544\\
23.4383366964924	38.0341188162726\\
23.4383366964924	-38.0341188162726\\
0.225704957176141	-22.5693375436862\\
0.225704957176141	22.5693375436862\\
8.22854433107183	-76.9916975553049\\
8.22854433107183	76.9916975553049\\
42.2844621341867	200.278598768553\\
42.2844621341867	-200.278598768553\\
20.3752452672412	-196.501463099846\\
20.3752452672412	196.501463099846\\
12.2530938904645	-306.541975572987\\
12.2530938904645	306.541975572987\\
12.2729972175068	-306.542630329882\\
12.2729972175068	306.542630329882\\
0.225705262947512	-22.5693369030295\\
0.225705262947512	22.5693369030295\\
8.23862557691978	-76.8566551146636\\
8.23862557691978	76.8566551146636\\
81.4056165292308	-174.831592825241\\
81.4056165292308	174.831592825241\\
19.7449057822091	-196.678044689627\\
19.7449057822091	196.678044689627\\
12.2723847530898	-306.541025245496\\
12.2723847530898	306.541025245496\\
19.7612631052846	-196.616780354884\\
19.7612631052846	196.616780354884\\
0.225704704777192	-22.5693375229508\\
0.225704704777192	22.5693375229508\\
86.1552070655277	-78.1778163583013\\
86.1552070655277	78.1778163583013\\
8.14881547900792	-76.8560635628295\\
8.14881547900792	76.8560635628295\\
12.2718334121096	-306.541820017045\\
12.2718334121096	306.541820017045\\
19.7313698847383	-196.608480773993\\
19.7313698847383	196.608480773993\\
0.22570481413169	-22.569337532756\\
0.22570481413169	22.569337532756\\
8.11160149971707	-76.8742560252325\\
8.11160149971707	76.8742560252325\\
46.9921653072793	-91.7128078952023\\
46.9921653072793	91.7128078952023\\
12.2719385942731	-306.541277190254\\
12.2719385942731	306.541277190254\\
19.7507075564417	-196.604606282483\\
19.7507075564417	196.604606282483\\
0.225704755391261	-22.5693376052891\\
0.225704755391261	22.5693376052891\\
8.10734707480891	-76.8407710124477\\
8.10734707480891	76.8407710124477\\
43.0933887050208	-80.8336814979849\\
43.0933887050208	80.8336814979849\\
12.2718911873275	-306.541272551711\\
12.2718911873275	306.541272551711\\
19.7510457586111	-196.604630303019\\
19.7510457586111	196.604630303019\\
0.225704791763891	-22.5693375776781\\
0.225704791763891	22.5693375776781\\
8.08654624168239	-76.818393658407\\
8.08654624168239	76.818393658407\\
35.3309126516883	-79.5212125874178\\
35.3309126516883	79.5212125874178\\
12.2719541250285	-306.541150623159\\
12.2719541250285	306.541150623159\\
19.7545744699619	-196.606560437447\\
19.7545744699619	196.606560437447\\
0.225704807960986	-22.5693375625889\\
0.225704807960986	22.5693375625889\\
8.09190150319003	-76.7892971354429\\
8.09190150319003	76.7892971354429\\
32.0772594211817	-76.3209185422538\\
32.0772594211817	76.3209185422538\\
12.2719737720129	-306.541102155962\\
12.2719737720129	306.541102155962\\
19.7557264044872	-196.607501488312\\
19.7557264044872	196.607501488312\\
0.225704813773346	-22.5693375458266\\
0.225704813773346	22.5693375458266\\
8.09335473418243	-76.7646983864342\\
8.09335473418243	76.7646983864342\\
30.0678890074453	-74.8064000430963\\
30.0678890074453	74.8064000430963\\
12.2719975882398	-306.541070545129\\
12.2719975882398	306.541070545129\\
19.7563618157751	-196.608530965663\\
19.7563618157751	196.608530965663\\
0.2257048121934	-22.5693375371782\\
0.2257048121934	22.5693375371782\\
8.09834492174192	-76.7485494236466\\
8.09834492174192	76.7485494236466\\
29.2225623371115	-73.740803736578\\
29.2225623371115	73.740803736578\\
12.2720112025976	-306.541057313634\\
12.2720112025976	306.541057313634\\
19.7565510305582	-196.609272914441\\
19.7565510305582	196.609272914441\\
0.225704804694406	-22.5693375380258\\
0.225704804694406	22.5693375380258\\
8.10103273405736	-76.7382424444509\\
8.10103273405736	76.7382424444509\\
28.8399703734877	-73.0593007326051\\
28.8399703734877	73.0593007326051\\
12.2720204427969	-306.541054195369\\
12.2720204427969	306.541054195369\\
19.7564431168236	-196.609920244266\\
19.7564431168236	196.609920244266\\
0.225704796060164	-22.5693375425774\\
0.225704796060164	22.5693375425774\\
8.10326587039652	-76.7317186913082\\
8.10326587039652	76.7317186913082\\
28.6920939452733	-72.542848363874\\
28.6920939452733	72.542848363874\\
};
\addlegendentry{intermediate};

\addplot [color=red,only marks,mark=diamond,mark options={solid}]
  table[row sep=crcr]{%
12.2720263461011	-306.541055323123\\
12.2720263461011	306.541055323123\\
19.7562438589507	-196.610439317929\\
19.7562438589507	196.610439317929\\
0.225704787568642	-22.5693375485581\\
0.225704787568642	22.5693375485581\\
8.10482356872348	-76.7272343740742\\
8.10482356872348	76.7272343740742\\
28.6318239659765	-72.1440416741725\\
28.6318239659765	72.1440416741725\\
};
\addlegendentry{final};

\end{axis}
\end{tikzpicture}%
					\caption{\corr{Points at which data of the FOM is collected during IRKA}.}
					\label{fig:IRKAshifts}
				\end{subfigure}
				\hfill
				\begin{subfigure}[t]{0.48\textwidth}
%
\definecolor{mycolor1}{rgb}{0.00000,0.44700,0.74100}%
\begin{tikzpicture}

\begin{axis}[%
width=0.95092\mywidth,
height=\myheight,
at={(0\mywidth,0\myheight)},
scale only axis,
xmin=0,
xmax=30,
xlabel={$i$},
ymode=log,
ymin=1e-20,
ymax=1,
yminorticks=true,
ylabel={$\varsigma_i/\varsigma_1$},
title style={font=\bfseries},
title={Singular value decay},
legend style={font=\footnotesize}
]
\addplot [color=mycolor1,solid,mark=o,mark options={solid},forget plot]
  table[row sep=crcr]{%
1	1\\
2	0.48028625039849\\
3	0.34105683691047\\
4	0.27700600542327\\
5	0.0151595068873339\\
6	0.00209307258733632\\
7	3.44476516008839e-05\\
8	2.40525120206174e-05\\
9	1.30005045637117e-06\\
10	6.66711433145258e-07\\
11	3.85618506986417e-07\\
12	1.84679319891831e-08\\
13	2.95701371623099e-10\\
14	1.34456331272881e-10\\
15	5.39069786717619e-11\\
16	9.50103984414813e-14\\
17	6.49302173701959e-14\\
18	3.90294449118889e-14\\
19	3.55238083472643e-14\\
20	2.20389208372869e-14\\
21	1.67782342606868e-14\\
22	7.0300925141478e-15\\
23	2.61763510877098e-15\\
24	1.24888702652983e-15\\
25	8.605913226954e-16\\
26	7.49418180073202e-16\\
};
\end{axis}
\end{tikzpicture}%
					\caption{\corr{Decay of singular values of the matrix $\left[\Loew,\sLoew\right]$ for the data collected during IRKA.}}
					\label{fig:LoewnerDecay}
				\end{subfigure}
				\caption{\corr{Data collecting during IRKA can be used to generate data-driven surrogates.}}
			\end{figure}
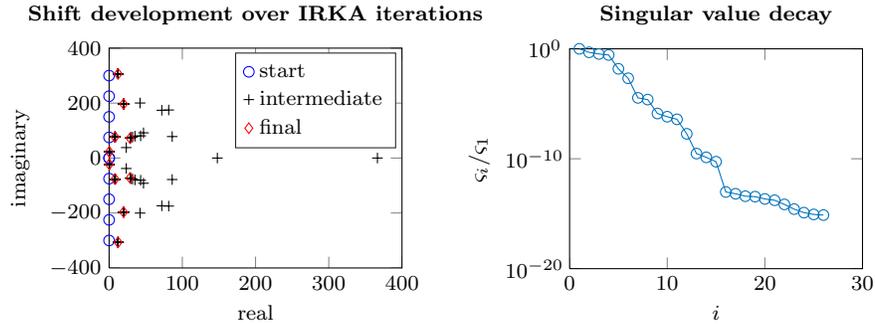
						
			To use this ``free'' data, there are various choices for ``data-driven" procedures that produce useful rational approximations.  \emph{Loewner methods} \cite{mayo2007framework, antoulas1986scalar, anderson1990rational, lefteriu2010new} are effective and are already integrated into IRKA iteration strategies \cite{beattie2012realization}.  We adopt here a \emph{vector fitting} strategy \cite{gustavsen1999rational,gustavsen2006improving,deschrijver2008macromodeling,drmac2015vectorrohtua, drmac2015quadrature} instead.  This allows us to produce \corr{stable} low-order approximations of the reduction error \corr{after IRKA}
			\begin{equation}
				\widetilde{G_e^0} \approx G_e^0\defeq G-\GrZ.
			\end{equation}  
			An appropriate choice of order for the surrogate model can be obtained by forming the Loewner $\Loew$ and shifted Loewner $\sLoew$ matrices 
			\textcolor{black}{from $G$ and $G'$ evaluations that were generated in the course of the IRKA iteration}  and then observing the singular value decay of the matrix $\left[\Loew,\sLoew\right]$, as indicated in Figure \ref{fig:LoewnerDecay}.

			Using the decomposition in \eqref{eq:GrDSplit}, the $\Hinf$-norm evaluations required during the optimization will be feasible even for large-scale full order models. In addition, it will allow us to obtain a cheap estimate $\tilde{e}_{\Hinf}$ for the approximation error
			\begin{equation}
				e_{\Hinf} \defeq \norm{G-\GrD}_{\Hinf} \approx \norm{\widetilde{G_e^0} - \Delta\GrD}_{\Hinf} = \tilde{e}_{\Hinf} 
			\end{equation}

		\subsubsection{Constrained multivariate optimization with respect to $D_r$}
	The focus of this work lies in the development of new model reduction strategies.  Our intent is not directed toward making a contribution to either the theory or practice of numerical optimization and we are content in this work to use standard optimization approaches.   In the results of section \ref{sec:Results}, we rely on state-of-the-art algorithms that are widespread and available, e.g., in \textsc{matlab}. 
			With that caveat understood, we do note that the constrained multivariate optimization over the reduced feed-through, $D_r$, is a challenging optimization problem, so we will explain briefly the setting that seems to work best in our case.  The computation and optimization of $\Hinf$-norms for large-scale models remains an active area of research, as demonstrated by \cite{mitchell2015fixed,mitchell2015hybrid,aliyev2016large}.
			
			The problem we need to solve in step \corr{2} of Algorithm \ref{algo:Hinf} is			
			\begin{equation}\label{eq:optimization}
			\begin{split}
			\min_{D_r\in\Reals^{p\times m}}&\max_{\omega}{\varsigma_{max}\left(G(j\omega) - \GrD(j\omega,\Dr)\right)}\\
			s.t. \quad &\GrD(s,\Dr) \; \text{is stable}
			\end{split}
			\end{equation}	
			which represents a non-smooth, non-convex multivariate optimization problem in a \corr{$p\!\times\! m$}-dimensional search space. In our experience, the best strategy considering both optimization time and optimal solution is given by a combination of \emph{coordinate descent} \corr{(CD)} \cite{wright2015coordinate} and subsequent \corr{\emph{multivariate optimization}} \corr{(MV)}.
			\corr{We refer to this combined strategy as CV+MV.}
			The coordinate descent strategy is used in this setting somewhat like an initialization procedure to find a better starting point than $\DrZ=0$. This initialization is based on reducing the search space from $p\cdot m$ dimensions to just one, hence performing a much simpler univariate optimization in each step. Once one cycle has been conducted for all elements in the feed-through matrix, the resulting feed-through is used to initialize a nonlinear constrained optimization solver that minimizes the error with respect to the whole $D_r$ matrix.  We have used a sequential quadratic programming (SQP) method as implemented in \textsc{matlab}'s \mcode{fmincon}, although acceptable options for this final step abound.   Further information about optimization strategies can be found in \cite{nocedal2006numerical}.		
			
			The suitability of \corr{CD+MV} is motivated by extensive simulations conducted comparing different strategies, such as direct multivariate optimization, \emph{global search} (GS) \cite{ugray2007scatter}, and \emph{genetic algorithms} (GA) \corr{(cp. Figure \ref{fig:SolverComparison})}.
			Ultimately, we rely on the results of Section \ref{sec:Results} to show that this procedure is effective.
			\setlength\mywidth{90 mm}
			\setlength\myheight{25 mm}
			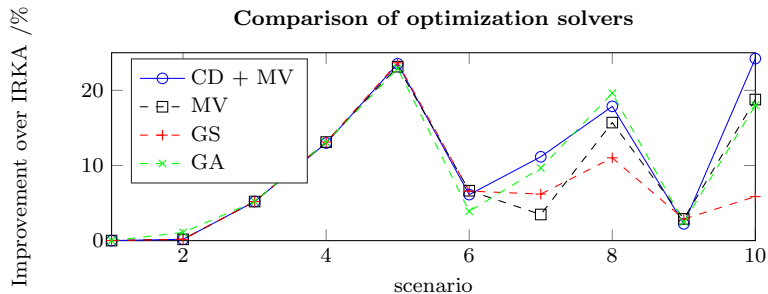
\begin{figure}[h]
				\centering
%
\begin{tikzpicture}

\begin{axis}[%
width=0.95092\mywidth,
height=\myheight,
at={(0\mywidth,0\myheight)},
scale only axis,
xmin=1,
xmax=10,
xlabel={scenario},
ymin=0,
ymax=25,
ylabel={Improvement over IRKA /\%},
title style={font=\bfseries},
title={Comparison of optimization solvers},
legend style={at={(0.03,0.97)},anchor=north west,legend cell align=left,align=left,draw=white!15!black},
legend style={font=\footnotesize}
]
\addplot [color=blue,solid,mark=o,mark options={solid}]
  table[row sep=crcr]{%
1	0.00681066951213616\\
2	0.175807859925581\\
3	5.20553864701622\\
4	12.9855059456818\\
5	23.5318408069326\\
6	6.12217995155312\\
7	11.1562249309093\\
8	17.8669521894189\\
9	2.23158120086834\\
10	24.2272177739557\\
};
\addlegendentry{CD + MV};

\addplot [color=black,dashed,mark=square,mark options={solid}]
  table[row sep=crcr]{%
1	0.00681064273571064\\
2	0.175807890469382\\
3	5.20439761978322\\
4	13.1171881141933\\
5	23.144980445202\\
6	6.64596734114894\\
7	3.48096333210287\\
8	15.7096492388692\\
9	2.89596780059834\\
10	18.7701117890886\\
};
\addlegendentry{MV};

\addplot [color=red,dashed,mark=+,mark options={solid}]
  table[row sep=crcr]{%
1	0.00681066725315427\\
2	0.17580790043441\\
3	5.20431871259779\\
4	13.1200027742559\\
5	23.6645659508247\\
6	6.59078061589113\\
7	6.19180733577567\\
8	11.0161428447163\\
9	2.86946372612744\\
10	5.86217452297733\\
};
\addlegendentry{GS};

\addplot [color=green,dashed,mark=x,mark options={solid}]
  table[row sep=crcr]{%
1	0.00681067277676917\\
2	1.09569512238805\\
3	5.20970632474869\\
4	13.0950118806022\\
5	22.8950749844458\\
6	3.90684635076182\\
7	9.64931393596479\\
8	19.6279658219395\\
9	2.47435733889919\\
10	18.0079489983\\
};
\addlegendentry{GA};

\end{axis}
\end{tikzpicture}%
				\caption{Comparison of different solvers shows the effectiveness of coordinate descent followed by multivariate optimization.}
				\label{fig:SolverComparison}
			\end{figure}			

	\section{Numerical results}\label{sec:Results}
	\corr{In the following we demonstrate the effectiveness of the proposed procedure by showing reduction results with different MIMO models. The reduction code is based on the sssMOR toolbox\footnote{Available at \url{www.rt.mw.tum.de/?sssMOR}.} \cite{castagnotto2016sss}.  For generation of vector fitting surrogates, we use the 
 \mcode{vectfit3} function\footnote{Available at \url{www.sintef.no/projectweb/vectfit/downloads/vfut3/}.} \cite{gustavsen1999rational,gustavsen2006improving,deschrijver2008macromodeling}. Note that more recent implementation of MIMO vector fitting introduced in \cite{drmac2015vectorrohtua} could be used instead, especially for improved robustness.}
	\subsection{Heat model}
		Our proposed procedure is demonstrated through numerical examples conducted on a MIMO benchmark model representing a discretized heat equation of order $\fo=197$ \corr{with $p=2$ outputs and $m=2$ inputs} \corr{\cite{antoulas2001survey}}.
		
		Model reduction for this model was conducted for a range of reduced orders; the results are summarized in \corr{Table \ref{tab:heatmodel}}.		
		The table shows the reduced order $\ro$, the order $\nm$ of the error surrogate $\widetilde{\GeZ}$, and \corr{the relative $\Hinf$ error of the proposed ROM $\GrD$,} as well as the percentage improvement over the initial IRKA model. 
		\textcolor{black}{Our proposed 
		method improves significantly on the $\mathcal{H}_\infty$ performance of IRKA, in some cases by more than $50\%$.}
		


\begin{table}[h!] 
	\centering
	\caption{Results for the heat model problem}
	\begin{tabular}{c|cccccccccc}                                            \hline                                
		$\ro$ & 1 & 2 & 3 & 4 & 5 & 6 & 7 & 8 & 9 & 10 \\       
		\hline                                                         
			$\nm$  & 14 & 24 & 20 & 22 & 24 & 30 & 32 & 36 & 36 & 36 \\                                                           
		$\frac{\norm{G-\GrD}}{\norm{G}}$ & 8.7e-02 & 7.6e-03 & 1.2e-02 & 1.2e-03 & \corr{6.5e-04} & 5.7e-04 & 4.1e-04 & 1.6e-04 & 4.4e-05 & 8.6e-06 \\ [2ex]
		$1-\frac{\norm{G-\GrD}}{\norm{G-\GrZ}}$ & 50.8\% & 39.0\% & 27.0\% & 36.7\% & \corr{36.0}\% & 44.8\% & 52.0\% & 44.6\% & 49.5\% & 42.6\% \\                        
	\end{tabular}                                                  	\label{tab:heatmodel}                                                                                  
\end{table}      
		
		Figure \ref{fig:heatmodel_comparison} gives a graphical representation of the reduction results. 
		The plots compare the approximation error achieved \corr{after applying MIHA, with a vector fitting surrogate as described in Section \ref{sec:surrogate},} to other reduction strategies. These include the direct reduction with IRKA, balanced truncation (BT), Optimal Hankel Norm \corr{A}pproximation (OHN\corr{A}) as well as the optimization with respect to the actual error $\GeZ$ \corr{(MIHA without surrogate)}. For a better graphical comparison throughout the reduced orders studied, the errors are related to the theoretical lower bound given by
		\begin{equation}
			\underline{e}_{\Hinf} \defeq \varsigma^H_{\ro+1},
		\end{equation}
		with which we denote the Hankel singular value of order $\ro+1$.
		\setlength\mywidth{90 mm}
		\setlength\myheight{24 mm}
		\begin{figure}[h!]
			\centering
%
\definecolor{mycolor1}{rgb}{0.00000,1.00000,1.00000}%
\begin{tikzpicture}

\begin{axis}[%
width=0.713813\mywidth,
height=\myheight,
at={(0\mywidth,0\myheight)},
scale only axis,
xmin=0,
xmax=10,
xlabel={reduced order},
ymin=1,
ymax=5,
ylabel={$\frac{\|G-G_r\|_{\mathcal{H}_\infty}}{\underline{e}_{\mathcal{H}_\infty}}$},
legend style={at={(1.03,0.5)},anchor=west,legend cell align=left,align=left,draw=white!15!black},
legend style={font=\footnotesize}
]
\addplot [color=blue,solid,mark=o,mark options={solid}]
  table[row sep=crcr]{%
1	1.27593678987847\\
2	1.05022077379947\\
3	1.90094672781014\\
4	1.08672610468601\\
5	1.17525630829163\\
6	1.13576197159972\\
7	2.32653656021203\\
8	1.12204470536217\\
9	1.29022027356565\\
10	1.1233247176539\\
};
\addlegendentry{MIHA (VF surr.)};

\addplot [color=blue,dash pattern=on 1pt off 3pt on 3pt off 3pt]
  table[row sep=crcr]{%
1	1.27192177588262\\
2	1.04486061333364\\
3	1.90365050543942\\
4	1.08270881613349\\
5	1.12704262657334\\
6	1.13322851709502\\
7	2.32104186213531\\
8	1.34319761988163\\
9	1.29001739767631\\
10	1.08457240915385\\
};
\addlegendentry{MIHA (no surr.)};

\addplot [color=red,solid,mark=square,mark options={solid}]
  table[row sep=crcr]{%
1	2.59547160394697\\
2	1.72232567703345\\
3	2.60575058010256\\
4	1.71714908655505\\
5	1.83707829918636\\
6	2.05784250026551\\
7	4.84878757527703\\
8	2.0254351611853\\
9	2.55643292538386\\
10	1.95759923986102\\
};
\addlegendentry{IRKA};

\addplot [color=black,dashed,mark=+,mark options={solid}]
  table[row sep=crcr]{%
1	1.16166586433089\\
2	1.06023859721613\\
3	1.0895098599177\\
4	1.18215718841685\\
5	1.21677169033204\\
6	1.0390773773994\\
7	1.06163236571238\\
8	1.0220577563731\\
9	1.11379572744826\\
10	1.0860415760103\\
};
\addlegendentry{OHNA};

\addplot [color=mycolor1,dashed,mark=triangle,mark options={solid,rotate=180}]
  table[row sep=crcr]{%
1	2.05739932979069\\
2	1.72470815071919\\
3	2.1528728463892\\
4	1.32951949668134\\
5	2.26091447521957\\
6	1.82925310314691\\
7	2.09675653805782\\
8	1.9201929059276\\
9	1.9779611164241\\
10	1.97539560240851\\
};
\addlegendentry{BT};

\end{axis}
\end{tikzpicture}%
			\caption{Plot of the approximation error relative to the theoretical error bound.}
			\label{fig:heatmodel_comparison}
		\end{figure}
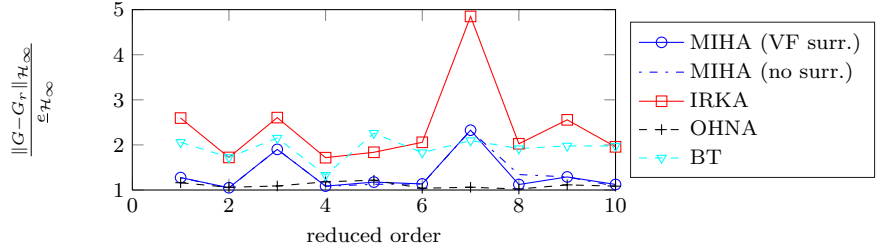	
		
		Notice how effectively the ROMs resulting from the $\Dr$-optimization reduce the $\Hinf$-error 
		beyond what is produced by the IRKA ROMs and that they often, 
		\textcolor{black}{(here, in \corr{$9$} out of $10$ cases)}
		yield better results than BT and sometimes \textcolor{black}{(here, in $3$ out of $10$ cases)} yield better results even than OHN\corr{A}.
		Note also that the optimization with respect to the vector-fitting surrogate produces as good a result as optimization with respect to the true error. For reduced order $\ro=8$, optimization with respect to the surrogate yields even a better result. This is not expected and may be due to the different cost functions involved, causing optimization of the true error to converge to a worse \corr{solution}.
		
		The plot also confirms our initial motivation in using IRKA models as starting points, since their approximation in terms of the $\Hinf$ norm is often not far from BT.
		Finally, note how in several cases the resulting ROM is very close to the theoretical lower bound, which implies that the respective ROMs are not far from being the \emph{global} optimum.		
		
		Figure \ref{fig:heatmodel_errorplot} shows the approximation error before and after the feed-through optimization for a selected reduced order of 2.
		The largest singular value is drastically reduced (ca. 40\%) by lifting up the value at high frequencies. This confirms our intuition that the $\Hinf$-optimal reduced order model should have a \corr{nearly constant error modulus} over all frequencies.
		Finally, Figure \ref{fig:heatmodel_errorEstimation} demonstrates the validity of the error estimate $\tilde{e}_{\Hinf}$ obtained using the surrogate model.
		\setlength\mywidth{0.35\textwidth}
		\setlength\myheight{0.25\textwidth} 
		\begin{figure}[h]
			\centering
			\begin{subfigure}[t]{0.48\textwidth}
				\input{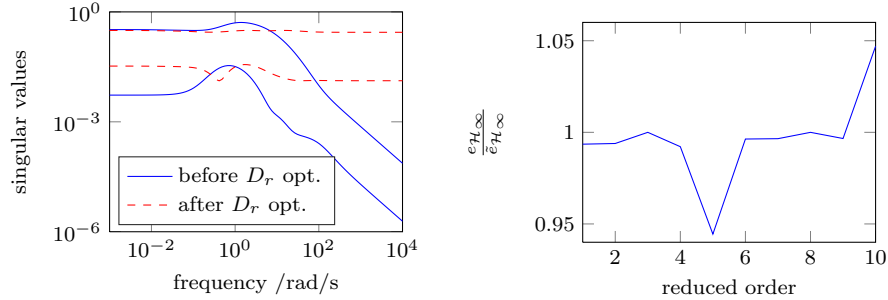}
				\caption{\corr{Singular value plot of the error} before and after optimization.}
				\label{fig:heatmodel_errorplot}
			\end{subfigure}
			\hfill
			\begin{subfigure}[t]{0.48\textwidth}
%
\begin{tikzpicture}

\begin{axis}[%
width=0.95092\mywidth,
height=\myheight,
at={(0\mywidth,0\myheight)},
scale only axis,
xmin=1,
xmax=10,
xlabel={reduced order},
ymin=0.94,
ymax=1.06,
ylabel={$\frac{e_{\mathcal{H}_\infty}}{\tilde{e}_{\mathcal{H}_\infty}}$},
legend style={font=\footnotesize}
]
\addplot [color=blue,solid,forget plot]
  table[row sep=crcr]{%
1	0.993536685037576\\
2	0.99390963319315\\
3	1.00000000917165\\
4	0.992160936116299\\
5	0.944445871456254\\
6	0.996343802795901\\
7	0.996539557439754\\
8	0.999995411853368\\
9	0.996630110836096\\
10	1.0469898078952\\
};
\end{axis}
\end{tikzpicture}%
				\caption{Comparison of error estimate $\tilde{e}_{\Hinf}$ versus true errror $e_{\Hinf}$.}
				\label{fig:heatmodel_errorEstimation}
			\end{subfigure}
			\caption{\corr{Optimization with the surrogate effectively reduces and provides an accurate estimate of the true error.}}
		\end{figure}

		\subsection{\corr{ISS model}}
		\corr{Similar simulations were conducted on a MIMO model with $m=3$ inputs and $p=3$ outputs of order $\fo = 270$, representing the 1r component of the International Space Station (ISS) \cite{Chahlaoui_Benchmark}. The results are summarized in Table \ref{tab:iss} and Figure \ref{fig:ISS_comparison}}.
		%
		\begin{table}                                                                                                    
			\centering                                                                                                       \caption{Results for the ISS problem}                        
			\begin{tabular}{c|cccccccccc}                                                                                     
				\hline $\ro$ & 2 & 4 & 6 & 8 & 10 & 12 & 14 & 16 & 18 & 20 \\                                                               
				\hline $\nm$  & 12 & 18 & 12 & 18 & 18 & 15 & 42 & 48 & 30 & 30 \\                                                          
				$\frac{\norm{G-\GrD}}{\norm{G}}$ &2.7e-01 & 9.4e-02 & 8.4e-02 & 7.9e-02 & 3.6e-02 & 3.4e-02 & 2.2e-02 & 2.2e-02 & 1.0e-02 & 7.7e-03 \\ [2ex]
				$1-\frac{\norm{G-\GrD}}{\norm{G-\GrZ}}$  & 7.5 \% & 9.9\% & 8.8\% & 4.9\% & 9.5\% & 13.8\% & 23.3\% & 15.7\% & 3.5\% & 25.8\% \\                              
			\end{tabular}  
			\label{tab:iss}                                                            
		\end{table}    
		\corr{Note that the $\Hinf$-error after IRKA is comparable to that of BT and the proposed procedure is effective in further reducing the error, outperforming BT in all cases investigated.}
		\setlength\mywidth{90 mm}
		\setlength\myheight{20 mm}		
		\begin{figure}[h!]
			\centering
%
\definecolor{mycolor1}{rgb}{0.00000,1.00000,1.00000}%
\begin{tikzpicture}

\begin{axis}[%
width=0.713813\mywidth,
height=\myheight,
at={(0\mywidth,0\myheight)},
scale only axis,
xmin=0,
xmax=20,
xlabel={reduced order},
ymin=1.4,
ymax=2.1,
ylabel={$\frac{\|G-G_r\|_{\mathcal{H}_\infty}}{\underline{e}_{\mathcal{H}_\infty}}$},
title style={font=\bfseries},
title={Error norms with respect to the theoretical lower bound},
legend style={at={(1.03,0.5)},anchor=west,legend cell align=left,align=left,draw=white!15!black},
legend style={font=\footnotesize}
]
\addplot [color=blue,solid,mark=o,mark options={solid}]
  table[row sep=crcr]{%
2	1.85062159355093\\
4	1.80330169504364\\
6	1.82528721046992\\
8	1.87725181174351\\
10	1.77263140010921\\
12	1.7547645712586\\
14	1.5505236854814\\
16	1.68659467349564\\
18	1.93865426715746\\
20	1.47676566755142\\
};
\addlegendentry{MIHA (VF surr.)};

\addplot [color=red,solid,mark=square,mark options={solid}]
  table[row sep=crcr]{%
2	2.00017913037851\\
4	2.0009010533358\\
6	2.00154070346904\\
8	1.97476519702429\\
10	1.95775942767939\\
12	2.03532118199623\\
14	2.02085673768045\\
16	1.99991997378079\\
18	2.00912370241388\\
20	1.99142000828074\\
};
\addlegendentry{IRKA};

\addplot [color=mycolor1,dashed,mark=triangle,mark options={solid,rotate=180}]
  table[row sep=crcr]{%
2	2.00019411320307\\
4	2.00090105336438\\
6	2.00154079403855\\
8	1.97470028763184\\
10	1.95766959318082\\
12	1.99971652976781\\
14	2.04718684090777\\
16	2.00000241435769\\
18	2.00910856605433\\
20	1.99322908548432\\
};
\addlegendentry{BT};

\end{axis}
\end{tikzpicture}%
			\caption{Plot of the approximation error relative to the theoretical error bound (ISS).}
			\label{fig:ISS_comparison}
		\end{figure}
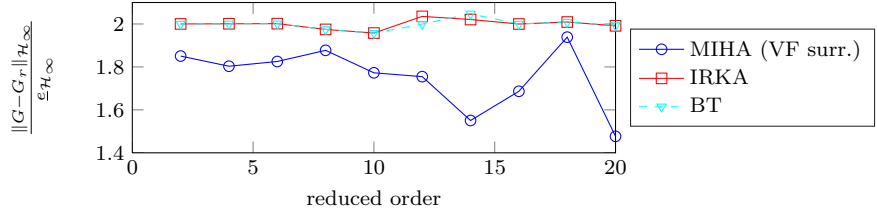
		\corr{Finally, note also in this case that the modulus of the error due to this $\Hinf$-approximation procedure is nearly constant, as anticipated. This is demonstrated in Figure \ref{fig:iss_errorplot}, where the error plots for the reduction order $\ro=10$ are compared.	}	
		\setlength\mywidth{90 mm}
		\setlength\myheight{30 mm}
		\begin{figure}[h!]
			\input{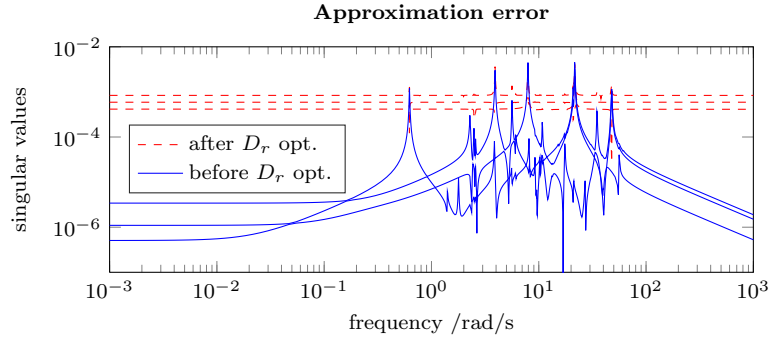}
			\caption{Singular value plot of the error before and after optimization (ISS).}
			\label{fig:iss_errorplot}
		\end{figure}

       §§
                 
\bibliographystyle{unsrt} 
\bibliography{Hinf}

\begin{thebibliography}{10}

\bibitem{Flagg_2013_Hinf}
Garret~M Flagg, C.~A. Beattie, and Serkan Gugercin.
\newblock Interpolatory $\mathcal{H}_\infty$ model reduction.
\newblock {\em Systems \& Control Letters}, 62(7):567--574, 2013.

\bibitem{Antoulas_Book}
A.~C. Antoulas.
\newblock {\em Approximation of Large-Scale Dynamical Systems}.
\newblock SIAM, 2005.

\bibitem{Gallivan_2002}
Kyle~A. Gallivan, Antoine Vandendorpe, and Paul Van~Dooren.
\newblock On the generality of multipoint pad{\'e} approximations.
\newblock In {\em 15th IFAC World Congress on Automatic Control}, 2002.

\bibitem{Beattie_2014_Survey}
C.~A. Beattie and Serkan Gugercin.
\newblock Model reduction by rational interpolation.
\newblock Accepted to appear in Model Reduction and Approximation for Complex
  Systems, 2014.

\bibitem{Grimme_PhD}
Eric~J. Grimme.
\newblock {\em Krylov Projection Methods for Model Reduction}.
\newblock PhD thesis, Dep. of Electrical Eng., Uni. Illinois at Urbana
  Champaign, 1997.

\bibitem{Gallivan_2004_MIMO}
Kyle~A. Gallivan, A~Vandendorpe, and Paul Van~Dooren.
\newblock Model reduction of mimo systems via tangential interpolation.
\newblock {\em SIAM Journal on Matrix Analysis and Applications},
  26(2):328--349, 2004.

\bibitem{Gallivan_2004_Sylvester}
Kyle~A. Gallivan, Antoine Vandendorpe, and Paul Van~Dooren.
\newblock Sylvester equations and projection-based model reduction.
\newblock {\em Journal of Computational and Applied Mathematics},
  162(1):213--229, 2004.

\bibitem{Gugercin_2008_IRKA}
Serkan Gugercin, A.~C. Antoulas, and C.~A. Beattie.
\newblock $\mathcal{H}_2$ model reduction for large-scale linear dynamical
  systems.
\newblock {\em SIAM Journal on Matrix Analysis and Applications},
  30(2):609--638, 2008.

\bibitem{antoulas2002h}
AC~Antoulas and A~Astolfi.
\newblock $\mathcal{H}_{\infty}$-norm approximation.
\newblock {\em Unsolved problems in Mathematical Systems and Control Theory, V.
  Blondel and A. Megretski Editors}, pages 267--270, 2002.

\bibitem{helmersson1994model}
Anders Helmersson.
\newblock Model reduction using lmis.
\newblock 1994.

\bibitem{varga2001fast}
Andras Varga and Pablo Parrilo.
\newblock Fast algorithms for solving hinf-norm minimization problems.
\newblock In {\em Decision and Control, 2001. Proceedings of the 40th IEEE
  Conference on}, volume~1, pages 261--266. IEEE, 2001.

\bibitem{Kavranoglu_1993}
Davut Kavranoglu and Maamar Bettayeb.
\newblock Characterization of the solution to the optimal hinf model reduction
  problem.
\newblock {\em Systems \& Control Letters}, 20(2):99--107, 1993.

\bibitem{Glover_1984}
Keith Glover.
\newblock All optimal hankel-norm approximations of linear multivariable
  systems and their linf-error bounds†.
\newblock {\em International journal of control}, 39(6):1115--1193, 1984.

\bibitem{trefethen1981rational}
L.N. Trefethen.
\newblock {Rational Chebyshev approximation on the unit disk}.
\newblock {\em Numerische Mathematik}, 37(2):297--320, 1981.

\bibitem{gugercin2004survey}
Serkan Gugercin and Athanasios~C Antoulas.
\newblock A survey of model reduction by balanced truncation and some new
  results.
\newblock {\em International Journal of Control}, 77(8):748--766, 2004.

\bibitem{benner2004computing}
Peter Benner, Enrique~S Quintana-Ort{\'\i}, and Gregorio Quintana-Ort{\'\i}.
\newblock Computing optimal hankel norm approximations of large-scale systems.
\newblock In {\em Decision and Control, 2004. CDC. 43rd IEEE Conference on},
  volume~3, pages 3078--3083. IEEE, 2004.

\bibitem{Li_2000}
Jing-Rebecca Li.
\newblock {\em Model reduction of large linear systems via low rank system
  gramians}.
\newblock PhD thesis, Massachusetts Institute of Technology, 2000.

\bibitem{benner2014self}
Peter Benner, Patrick K{\"u}rschner, and Jens Saak.
\newblock Self-generating and efficient shift parameters in {ADI} methods for
  large {L}yapunov and {S}ylvester equations.
\newblock {\em Electronic Transactions on Numerical Analysis}, 43:142--162,
  2014.

\bibitem{kurschner2016efficient}
Patrick K{\"u}rschner.
\newblock {\em Efficient Low-Rank Solution of Large-Scale Matrix Equations}.
\newblock PhD thesis, Otto-von-Guericke Universit{\"a}t Magdeburg, 2016.

\bibitem{sabino2006solution}
John Sabino.
\newblock {\em Solution of large-scale {L}yapunov equations via the block
  modified Smith method}.
\newblock PhD thesis, Citeseer, 2006.

\bibitem{gugercin2003modified}
Serkan Gugercin, Danny~C Sorensen, and Athanasios~C Antoulas.
\newblock A modified low-rank smith method for large-scale {L}yapunov
  equations.
\newblock {\em Numerical Algorithms}, 32(1):27--55, 2003.

\bibitem{gustavsen1999rational}
Bj{\o}rn Gustavsen and Adam Semlyen.
\newblock Rational approximation of frequency domain responses by vector
  fitting.
\newblock {\em Power Delivery, IEEE Transactions on}, 14(3):1052--1061, 1999.

\bibitem{drmac2015vectorrohtua}
Z.~Drma{\v{c}}, S.~Gugercin, and C.~Beattie.
\newblock Vector fitting for matrix-valued rational approximation.
\newblock {\em SIAM Journal on Scientific Computing}, 37(5):A2346--A2379, 2015.

\bibitem{Chahlaoui_Benchmark}
Younes Chahlaoui and Paul Van~Dooren.
\newblock A collection of benchmark examples for model reduction of linear time
  invariant dynamical systems.
\newblock Working Note 2002-2, 02 2002.

\bibitem{mayo2007framework}
A.J. Mayo and A.~C. Antoulas.
\newblock A framework for the solution of the generalized realization problem.
\newblock {\em Linear Algebra and Its Applications}, 425(2--3):634--662, 2007.

\bibitem{beattie2009interpolatory}
C.~A. Beattie and Serkan Gugercin.
\newblock Interpolatory projection methods for structure-preserving model
  reduction.
\newblock {\em Systems \& Control Letters}, 58-3:225--232, 2009.

\bibitem{golub2012matrix}
Gene~H Golub and Charles~F Van~Loan.
\newblock {\em Matrix computations}, volume~3.
\newblock JHU Press, 2012.

\bibitem{antoulas1986scalar}
AC~Antoulas and BDQ Anderson.
\newblock On the scalar rational interpolation problem.
\newblock {\em IMA Journal of Mathematical Control and Information},
  3(2-3):61--88, 1986.

\bibitem{anderson1990rational}
BDO Anderson and AC~Antoulas.
\newblock Rational interpolation and state-variable realizations.
\newblock {\em Linear Algebra and its Applications}, 137:479--509, 1990.

\bibitem{lefteriu2010new}
Sanda Lefteriu and Athanasios~C Antoulas.
\newblock A new approach to modeling multiport systems from frequency-domain
  data.
\newblock {\em Computer-Aided Design of Integrated Circuits and Systems, IEEE
  Transactions on}, 29(1):14--27, 2010.

\bibitem{beattie2012realization}
C.~A. Beattie and Serkan Gugercin.
\newblock Realization-independent $\mathcal{H}_2$-approximation.
\newblock In {\em 51st IEEE Conference on Decision and Control}, pages
  4953--4958. IEEE, 2012.

\bibitem{gustavsen2006improving}
Bjorn Gustavsen.
\newblock Improving the pole relocating properties of vector fitting.
\newblock {\em Power Delivery, IEEE Transactions on}, 21(3):1587--1592, 2006.

\bibitem{deschrijver2008macromodeling}
Dirk Deschrijver, Michal Mrozowski, Tom Dhaene, and Daniel De~Zutter.
\newblock Macromodeling of multiport systems using a fast implementation of the
  vector fitting method.
\newblock {\em Microwave and Wireless Components Letters, IEEE},
  18(6):383--385, 2008.

\bibitem{drmac2015quadrature}
Z~Drmac, S~Gugercin, and C~Beattie.
\newblock Quadrature-based vector fitting for discretized h\_2 approximation.
\newblock {\em SIAM Journal on Scientific Computing}, 37(2):A625--A652, 2015.

\bibitem{mitchell2015fixed}
Tim Mitchell and Michael~L Overton.
\newblock Fixed low-order controller design and $\mathcal{H}_\infty$
  optimization for large-scale dynamical systems.
\newblock {\em IFAC-PapersOnLine}, 48(14):25--30, 2015.

\bibitem{mitchell2015hybrid}
Tim Mitchell and Michael~L Overton.
\newblock Hybrid expansion--contraction: a robust scaleable method for
  approximating the $\mathcal{H}_\infty$ norm.
\newblock {\em IMA Journal of Numerical Analysis}, page drv046, 2015.

\bibitem{aliyev2016large}
N.~Aliyev, P.~Benner, E.~Mengi, and M.~Voigt.
\newblock Large-scale computation of $\mathcal{H}_\infty$ norms by a greedy
  subspace method.
\newblock {\em Submitted to SIAM J. Matrix Anal. Appl.}, 2016.

\bibitem{wright2015coordinate}
Stephen~J Wright.
\newblock Coordinate descent algorithms.
\newblock {\em Mathematical Programming}, 151(1):3--34, 2015.

\bibitem{nocedal2006numerical}
Jorge Nocedal and Stephen Wright.
\newblock {\em Numerical optimization}.
\newblock Springer Science \& Business Media, 2006.

\bibitem{ugray2007scatter}
Zsolt Ugray, Leon Lasdon, John Plummer, Fred Glover, James Kelly, and Rafael
  Martí.
\newblock Scatter search and local nlp solvers: A multistart framework for
  global optimization.
\newblock {\em INFORMS Journal on Computing}, 19(3):328--340, 2007.

\bibitem{castagnotto2016sss}
Alessandro Castagnotto, Maria Cruz~Varona, Lisa Jeschek, and Boris Lohmann.
\newblock sss \& {sssMOR}: Analysis \& reduction of large-scale dynamic systems
  with {MATLAB}.
\newblock In Preparation.

\bibitem{antoulas2001survey}
A.~C. Antoulas, D.~C. Sorensen, and S.~Gugercin.
\newblock A survey of model reduction methods for large-scale systems.
\newblock {\em Structured Matrices in OperaStructured, Numerical Analysis,
  Control, Signal and Image Processing, Contemporary Mathematics, AMS
  publications}, 280:193--219, 2001.

\end{thebibliography}

\end{document}